\documentclass[12pt]{amsart}
\usepackage[all]{xy}

\usepackage{amsmath, amsthm, amssymb,latexsym, graphics}

\newcommand{\N}{\mathbb{N}}
\newcommand{\Z}{\mathbb{Z}}

\newcommand{\R}{\mathbb{R}}
\newcommand{\C}{\mathbb{C}}
\newcommand{\E}{\mathbb{E}}

\newcommand{\spr}[2]{\langle #1, #2 \rangle}

\newcommand{\Eai}{E^\alpha_\infty}
\newcommand{\Eau}{E^\alpha_{\text{unif}}}
\newcommand{\equi}{\phi}
\newcommand{\dyad}{\varphi}
\newcommand{\Fdyad}{\psi}
\newcommand{\ta}{\langle t \rangle}
\newcommand{\na}{\langle n \rangle}
\newcommand{\M}{\mathcal{M}}
\newcommand{\Ma}{\M^\alpha}
\newcommand{\Bes}{\mathcal{B}}

\DeclareMathOperator{\supp}{supp}
\DeclareMathOperator{\Id}{id}

\DeclareMathOperator{\Gauss}{Gauss}

\let\Re=\relax \DeclareMathOperator{\Re}{Re}
\let\Im=\relax \DeclareMathOperator{\Im}{Im}

\newcommand{\bignorm}[1]{\bigl\Vert#1\bigr\Vert}
\newcommand{\Bignorm}[1]{\Bigl\Vert#1\Bigr\Vert}

\newtheorem{thmalt}{Theorem}[section]

\theoremstyle{definition}
\newtheorem{rem}[thmalt]{Remark}
\newtheorem{defi}[thmalt]{Definition}
\newtheorem{thm}[thmalt]{Theorem}
\newtheorem{cor}[thmalt]{Corollary}
\newtheorem{lem}[thmalt]{Lemma}
\newtheorem{prop}[thmalt]{Proposition}

\numberwithin{equation}{section}

\title[Spectral multipliers for wave operators]
 {Spectral multipliers for wave operators} 

\author{Ch. Kriegler}
\address{Ch. Kriegler\\
Laboratoire de Math\'ematiques (CNRS UMR 6620)\\
Universit\'e Blaise-Pascal (Clermont-Ferrand 2)\\
Campus des C\'ezeaux\\
63177 Aubi\`ere Cedex\\
France
}
\email{christoph.kriegler@math.univ-bpclermont.fr}

\date{\today}

\allowdisplaybreaks

\begin{document}

\begin{abstract}
A classical theorem of Mihlin yields $L^p$ estimates for spectral multipliers $L^p(\mathbb{R}^d) \to L^p(\mathbb{R}^d),\: g \mapsto \mathcal{F}^{-1}[ f( | \cdot |^2) \cdot \hat{g} ],$
in terms of $L^\infty$ bounds of the multiplier function $f$ and its weighted derivatives up to an order $\alpha > \frac{d}{2}.$
This theorem, which is a functional calculus for the standard Laplace operator, has generalisations in several contexts such as elliptic operators on domains and manifolds,
Schr\"odinger operators and sublaplacians on Lie groups.
However, for the wave equation functions $f_\alpha(\lambda) = (1 + \lambda)^{-\alpha} e^{i t \lambda},$ a better estimate is available, in the standard case (works of Miyachi and Peral)
and on Heisenberg Lie groups (M\"uller and Stein).
By a transference method for polynomially bounded regularized groups, we obtain a new class of spectral multipliers for operators that have these better wave spectral multipliers and that admit a
spectral decomposition of Paley-Littlewood type.
\end{abstract}

\maketitle

\section{Introduction}\label{Sec 1 Intro}

This article treats spectral multiplier problems.
A classical example is Mihlin's theorem \cite{Mikh} telling that for a function $f : (0,\infty) \to \C$ the corresponding Fourier multiplier $L^p(\R^d) \to L^p(\R^d),\: g \mapsto \mathcal{F}^{-1} [ \hat{g} f(|\cdot|^2 ) ]$ is bounded for any $1 < p < \infty$ provided that
\begin{equation}\label{Equ Hor Intro}
\sup_{t > 0} t^k |f^{(k)}(t)| < \infty \quad (k = 0,1,\ldots,\alpha)
\end{equation}

where $\alpha > d/2.$
There are many generalisations of this result in the literature (see \cite{DuOS} and
the many references therein) associating to a function $f$ a spectral multiplier $f(A)$ acting on some Banach space $X,$ mostly $X = L^p(\Omega)$ for some $1 < p < \infty.$
In the classical case this becomes $A = - \Delta,$ $X = L^p(\R^d).$
Also the above condition \eqref{Equ Hor Intro} is refined to a norm $\|f\|_{\Ma}$ with 
a real parameter $\alpha > 0$ and associated Banach algebra $\Ma$
(definition in Section \ref{Sec 2 Prelims}).
A Banach space valued treatise of this issue can be found in \cite{Kr,KrW}.

In this article a refinement of the spectral multiplier problem is regarded.
The motivation is that for some cases, a certain wave spectral multiplier admits an estimate which is better than what gives Mihlin's result.
Namely, let $f_\alpha(\lambda) = (1 + \lambda)^{-\alpha} e^{it \lambda}.$
We write in short $\ta = 1 + |t|$ and $a \lesssim b$ for $\exists\:c:\: a \leq c b.$
Then $f_\alpha$ satisfies for any $\epsilon \in (0,\alpha),\:\|f_\alpha\|_{\M^{\alpha - \epsilon}} \lesssim \ta^{\alpha},$
which gives then estimates of the spectral multiplier $f_\alpha(A)$ on $L^p$ for $\alpha > \frac{d}{2}$ and $1 < p < \infty.$

Surprisingly, in some cases of operators $A,$ a better estimate of $f_\alpha(A)$ is available
than given by Mihlin's theorem.
Namely, in \cite{Per} for the classical case and in \cite[(3.1)]{Mu} for the case of a sublaplacian operator
on a Heisenberg group, it is proved that for the square root $A$ of $-\Delta$ resp. of the sublaplacian,
\begin{equation}\label{Equ Intro 1}
\|f_\alpha(A) \|_{p \to p} \lesssim \ta^{\alpha}
\end{equation}
with $\alpha > \frac{d-1}{2}$ and $1 < p < \infty,$ so the critical value of $\alpha$ is smaller by $\frac12.$
This observation is the starting point of the present article.

Apart from $\Ma,$ we introduce two new functional calculus classes $\Eai$ and $\Eau$.
The second one admits an embedding from and into $\M^\beta$ depending on what are the values of $\alpha$ and $\beta,$ whereas the first one can be nicely compared to Besov spaces $\Bes^\alpha_{\infty,1},$
see \cite[Proposition 3.5]{Kr1} where it is studied in detail.
By means of a transference principle, we show that a condition \eqref{Equ Intro 1} together with a second similar bound imply that $A$ which acts on some Banach space $X$ 
has a smoothed $\Eai$ calculus in the sense that
\[ \| ( 1 + A )^{-\beta} f(A) \| \leq C \| f \|_{\Eai} \]
for a certain power $\beta.$\\

One of the consequences of a Mihlin type theorem is that $A$ admits a spectral decomposition of Paley-Littlewood type.
By this we mean, that if $(\dyad_n)_{n \in \Z}$ is a dyadic partition of unity (see Definition \ref{Def Partitions}), then the norm on the space $X$ where $A$ acts on admits a partition of the form
\begin{equation}\label{Equ Paley-Littlewood}
\|x\|^2 \cong \E \| \sum_{n \in \Z} \gamma_n \otimes \dyad_n(A) x \|^2,
\end{equation}
where $\gamma_n$ are independent Gaussian random variables on some probability space.
The expression on the right hand side of \eqref{Equ Paley-Littlewood} is also used to define the notion of $\gamma$-boundedness well-known to specialists (see Section \ref{Sec 2 Prelims}).

A further result in this article is that if $A$ satisfies a strengthened $\gamma$-bounded version of \eqref{Equ Intro 1} together with a Paley-Littlewood decomposition
\eqref{Equ Paley-Littlewood}, then $A$ has an $\Eau$ functional calculus.
Furthermore, in Theorem \ref{Thm 4}, we obtain an equivalence of the strengthened $\gamma$-bounded form of \eqref{Equ Intro 1} and a $\gamma$-bounded functional calculus.
Secondly, we deduce the $\Eau$ calculus.

This theorem applies to the standard case, which is the content of Section \ref{Sec 5 Poisson Semigroup}.
There we prove that the hypothesis of Theorem \ref{Thm 4} is satisfied for $A = (-\Delta)^{\frac12}.$
Apart from an application of Theorem \ref{Thm 4}, we deduce a $\gamma$-bounded strengthening of the very first cited result, the classical Mihlin theorem.

\section{Preliminaries}\label{Sec 2 Prelims}

In this section, we present the tools used in the subsequent sections.

\begin{defi}\label{Def Partitions}
\begin{enumerate}
\item Let $\equi \in C^\infty_c$ such that $\supp \equi \subset [-1,1].$
Put $\equi_n = \equi(\cdot - n)$ and assume that $\sum_{n \in \Z} \equi_n(t) = 1$ for any $t \in \R.$
We call $(\equi_n)_n$ an equidistant partition of unity.
\item Let $\dyad \in C^\infty_c$ such that $\supp \dyad \subset [\frac12,2]$ and with $\dyad_n = \dyad(2^{-n} \cdot)$ we have $\sum_{n \in \Z} \dyad_n(t) = 1$ for any $t > 0,$
then we call $(\dyad_n)_n$ a dyadic partition of unity.
\item Let $\Fdyad_0,\,\Fdyad_1 \in C^\infty_c(\R)$ such that $\supp \Fdyad_1 \subset [\frac12,2]$ and $\supp \Fdyad_0 \subset [-1,1].$
For $n \geq 2,$ put $\Fdyad_n = \Fdyad_1(2^{1-n}\cdot),$ so that $\supp \Fdyad_n \subset [2^{n-2},2^n].$
For $n \leq -1,$ put $\Fdyad_n = \Fdyad_{-n}(-\cdot).$
We assume that $\sum_{n \in \Z} \Fdyad_n(t) = 1$ for all $t \in \R.$
Then we call $(\Fdyad_n)_{n \in \Z}$ a dyadic Fourier partition of unity, which we will exclusively use to decompose the Fourier image of a function.
\end{enumerate}
\end{defi}

For the existence of such smooth partitions, we refer to the idea in \cite[Lemma 6.1.7]{BeL}.
Whenever $(\phi_n)_n$ is a partition of unity as above, we put
\[
\widetilde\phi_n = \sum_{k=-1}^1 \phi_{n+k}.
\]
It is useful to note that
\[
\widetilde\phi_m \phi_n = \phi_n\text{ for }m=n\text{ and }\widetilde\phi_m \phi_n = 0\text{ for }|n-m| \geq 2.
\]

The Besov spaces $\Bes^\alpha_{\infty,\infty}$ and $\Bes^\alpha_{\infty,1},$ are defined for example in \cite[p. 45]{Triea}:
Let $(\Fdyad_n)_{n\in\Z}$ be a dyadic Fourier partition of unity.
Then
\[ \Bes^\alpha_{\infty,\infty} = \{ f \in C_b^0 : \: \|f\|_{B^\alpha_{\infty,\infty}} = \sup_{n \in \Z} 2^{|n|\alpha} \|f \ast \check{\Fdyad_n}\|_\infty < \infty\} \]
and
\[
 \Bes^\alpha_{\infty,1} = \{ f \in C_b^0 : \: \|f\|_{B^\alpha_{\infty,1}} = \sum_{n \in \Z} 2^{|n|\alpha} \|f \ast \check{\Fdyad_n}\|_\infty < \infty\}.
\]
Note that $\Bes^{\alpha}_{\infty,1} \hookrightarrow \Bes^{\alpha}_{\infty,\infty} \hookrightarrow \Bes^{\alpha-\epsilon}_{\infty,1}$ \cite[2.3.2. Proposition 2]{Triea}.
We define the Mihlin class for some $\alpha > 0$ to be
\[\Ma = \{ f : \R_+ \to \C:\: f_e \in \Bes^ \alpha_{\infty,1}\},\]
equipped with the norm $\|f\|_{\Ma} = \|f_e\|_{\Bes^\alpha_{\infty,1}}.$ 
Here we write
\[f_e : J \to \C,\,z \mapsto f(e^z)\]
for a function $f : I \to \C$ such that $I \subset \C \backslash (-\infty,0]$ and $J = \{ z \in \C : \: | \Im z | < \pi,\:e^z \in I \}.$
The space $\Ma$ coincides with the space $\Lambda^\alpha_{\infty, 1}(\R_+)$ in \cite[p. 73]{CDMY}.
We point out the particular function
\[ f_\alpha(\lambda) = (1 + \lambda)^{-\alpha} e^{it\lambda}.\]
The function $f_\alpha$ belongs to $\M^{\alpha - \epsilon}$ for any $\epsilon \in (0,\alpha)$ with $\|f_\alpha\|_{\M^{\alpha - \epsilon}} \leq C \ta^\alpha$ \cite[Proposition 4.12]{Kr}.

Let $(\gamma_k)_{k \geq 1}$ be a sequence of independent standard Gaussian variables on some probability space $\Omega_0.$
Then we let $\Gauss(X) \subset L^2(\Omega_0;X)$ be the closure of ${\rm Span}\{\gamma_k\otimes x\, :\, k\geq 1,\ x\in X\}$ in
$L^2(\Omega_{0};X)$.
For any finite family $x_1,\ldots,x_n$ in $X,$ we have
$$
\Bignorm{\sum_k \gamma_k\otimes x_k}_{\Gauss(X)} \,=\,
\left(\E \Bignorm{\sum_k \gamma_k(\cdot) x_k}_X^2\right)^{\frac12} \,=\,
\Bigr(\int_{\Omega_0}\Bignorm{\sum_k \gamma_k(\lambda)\,
x_k}_{X}^{2}\,d\lambda\,\Bigr)^{\frac{1}{2}}.
$$
Now let $\tau \subset B(X).$
We say that $\tau$ is $\gamma$-bounded if there is a constant $C\geq 0$ such that for any finite families
$T_1,\ldots, T_n$ in $\tau$, and $x_1,\ldots,x_n$ in $X$, we have
$$
\Bignorm{\sum_k \gamma_k\otimes T_k x_k}_{\Gauss(X)}\,\leq\, C\,
\Bignorm{\sum_k \epsilon_k\otimes x_k}_{\Gauss(X)}.
$$
In this case, we let $\gamma(\tau)$ denote the smallest possible $C$.
If $X$ is a Hilbert space then $\gamma(\tau) = \sup_{T \in \tau} \|T\|$ and in a general Banach space, $\gamma(\tau) \geq \sup_{T \in \tau} \|T\|.$
Note that Kahane's contraction principle states that $\tau = \{ c \Id_X :\: |c| \leq 1 \} \subset B(X)$ is $\gamma$-bounded for any Banach space $X.$
Recall that by definition, $X$ has Pisier's property $(\alpha)$ if for any finite family $x_{k,l}$ in $X,$ $(k,l) \in F,$ where $F \subset \Z \times \Z$ is a finite array,
we have a uniform equivalence
\[ \bignorm{ \sum_{(k,l) \in F} \gamma_k \otimes \gamma_l \otimes x_{k,l} }_{\Gauss(\Gauss(X))} \cong \bignorm{ \sum_{(k,l) \in F} \gamma_{k,l} \otimes x_{k,l} }_{\Gauss(X)}. \]
Examples of spaces with property $(\alpha)$ are subspaces of an $L^p$ space with $p < \infty.$ 

Let $H$ be a separable Hilbert space.
We consider the tensor product $H \otimes X$ as a subspace of $B(H,X)$ in the usual way, i.e. by identifying
$\sum_{k=1}^n h_k \otimes x_k \in H \otimes X$ with the mapping $u : h \mapsto \sum_{k=1}^n \spr{h}{h_k} x_k$ for any finite
families $h_1,\ldots,h_n \in H$ and $x_1,\ldots,x_n \in X.$
Choose such families with corresponding $u$, where the $h_k$ shall be orthonormal.
Let $\gamma_1,\ldots,\gamma_n$ be independent standard Gaussian random variables over some probability space.
We equip $H \otimes X$ with the norm
\[ \| u \|_{\gamma(H,X)} = \bignorm{ \sum_k \gamma_k \otimes x_k }_{\Gauss(X)}.\]
By \cite[Corollary 12.17]{DiJT}, this expression is independent of the choice of the $h_k$ representing $u.$
We let $\gamma(H,X)$ be the completion of $H \otimes X$ in $B(H,X)$ with respect to that norm.
Then for $u \in \gamma(H,X),$ $\|u\|_{\gamma(H,X)} = \bignorm{\sum_k \gamma_k \otimes u(e_k)}_{\Gauss(X)},$ where the $e_k$ form an orthonormal basis of $H$ \cite[Definition 3.7]{vN}.

Assume that $(\Omega,\mu)$ is a $\sigma$-finite measure space and $H = L^2(\Omega).$
Denote $P_2(\Omega,X)$ the space of Bochner-measurable functions $f: \Omega \to X$ such that $x' \circ f \in L^2(\Omega)$ for all $x' \in X'.$
We identify $P_2(\Omega,X)$ with a subspace of $B(L^2(\Omega),X'')$ by assigning to $f$ the operator $u_f$ defined by
\[
\spr{u_f h}{x'} = \int_\Omega \spr{f(t)}{x'} h(t) d\mu(t).
\]
An application of the uniform boundedness principle shows that, in fact, $u_f$ belongs to $B(L^2(\Omega),X)$ \cite[Section 4]{KaW2}, \cite[Section 5.5]{Frohl}.
Then we let
\[\gamma(\Omega,X) = \left\{ f \in P_2(\Omega,X):\: u_f \in \gamma(L^2(\Omega),X)\right\}\]
and set
\[\|f\|_{\gamma(\Omega,X)} = \|u_f\|_{\gamma(L^2(\Omega),X)}.\]
The space $\{u_f:\:f \in \gamma(\Omega,X)\}$ is a proper subspace of $\gamma(L^2(\Omega),X)$ in general.
It is dense in $\gamma(L^2(\Omega),X)$ as it contains $L^2(\Omega) \otimes X.$
An element in $\gamma(\Omega,X)$ is called square function.
For more reading on this subject we refer to \cite{vN} and for similar objects to \cite{AKMP}.
For a proof of the following lemma, we refer to \cite{vN}.

\begin{lem}\label{Lem gamma lifting}
\begin{enumerate}
\item If $K \in B(H_1,H_2)$ where $H_1$ and $H_2$ are Hilbert spaces and $u \in \gamma(H_2,X)$ then we have $u \circ K \in \gamma(H_1,X)$ and $\|u \circ K\|_{\gamma(H_1,X)} \leq \|u\|_{\gamma(H_2,X)} \|K\|.$
\item For $f \in \gamma(\R,X)$ and $g \in \gamma(\R,X'),$ we have
\[ \int_{\R} |\langle f(t),g(t) \rangle| dt \leq \|f\|_{\gamma(\R,X)} \|g\|_{\gamma(\R,X')}. \]
\end{enumerate}
\end{lem}

A closed operator $A : D(A) \subset X \to X$ is called $\omega$-sectorial, if the spectrum $\sigma(A)$ is contained in $\overline{\Sigma_\omega},$ $R(A)$ is dense in $X$ and
\begin{equation}\label{Equ Def Sectorial}
\text{for all }\theta > \omega\text{ there is a }C_\theta > 0\text{ such that }\|\lambda (\lambda - A)^{-1}\| \leq C_\theta \text{ for all }\lambda \in \overline{\Sigma_\theta}^c.
\end{equation}
Note that $\overline{R(A)} = X$ along with \eqref{Equ Def Sectorial} implies that $A$ is injective.
We are particularly interested in operators that are $\omega$-sectorial for any $\omega > 0$ and call them $0$-sectorial operator.
For such operators there is a theory of holomorphic functional calculus \cite{CDMY}.
Building upon this, the $0$-sectorial operator $A$ is said to have a Mihlin calculus, or more precisely a $\Ma$ calculus if there exists $C > 0$ such that
$\|f(A)\| \leq C \|f\|_{\Ma}$ for any $f \in \Ma$ \cite[Definition 4.17]{Kr}.

Any $0$-sectorial operator always generates a $C_0$-semigroup $\exp(-tA)$ which is analytic on the whole right half plane.
We have the following link between $\gamma$ bounds of this semigroup and of $f_\alpha(2^k A)$, with the function $f_\alpha$ as above.
Consider
\begin{equation}\label{Equ Semigroup Bound}
\gamma\left(\left\{ \exp(-e^{i\theta}2^k t A) : \: k \in \Z \right\} \right) \lesssim (\frac{\pi}{2} - |\theta|)^{-\alpha}
\end{equation}
and
\begin{equation}\label{Equ Wave Bound}
\gamma\left( \left\{ (1 + 2^k A)^{-\alpha} e^{it2^k A} :\: k \in \Z \right\} \right) \lesssim  \ta^\alpha.
\end{equation}
Then \eqref{Equ Semigroup Bound} $\Longrightarrow$ \eqref{Equ Wave Bound} \cite[Lemma 4.72]{Kr}.

\section{Smoothed $\Eai$ calculus}

\begin{defi}
Let $(\equi_n)_{n \in \Z}$  be an equidistant partition of unity.
We define for an $\alpha > 0$
\[ \Eai = \left\{ f : \R \to \C :\: \|f\|_{\Eai} = \sum_{n \in \Z} \na^\alpha \|f \ast \check\equi_n\|_\infty < \infty \right\}.\]
\end{defi}

Properties of this space are investigated in detail in \cite{Kr1}.

\begin{defi}
Let $(\equi_n)_{n \in \Z}$ be an equidistant partition of unity and $(\dyad_k)_{k \in \Z}$ a dyadic partition of unity.
Then we define for an $\alpha > 0$ 
\[ \Eau = \left\{ f : (0,\infty) \to \C :\: \|f\|_{\Eau} = \sum_{n \in \Z} \na^\alpha \sup_{k \in \Z} \| [f(2^k \cdot)\varphi_0] \ast \check\phi_n \|_\infty < \infty \right\}. \]
\end{defi}

The space $\Eau$ satisfies the following elementary properties.

\begin{lem}
\begin{enumerate}
\item The definition of $\Eau$ is independent of the choice of the dyadic partition $(\dyad_k)_k.$
\item $\Eau$ is an algebra, more precisely, if $f,g \in \Eau,$ then $\|f\cdot g\|_{\Eau} \leq C \|f\|_{\Eau} \|g\|_{\Eau}.$
\end{enumerate}
\end{lem}

\begin{proof}
The first part of the lemma is easy to check and left to the reader.
Now let $f,g \in \Eau.$
We write in the following in short $\sum_{l,j}^*$ for $\sum_{l,j:\: |n-(l+j)| \leq 3}.$
\begin{align*}
\| \left[f(2^k \cdot) g(2^k \cdot) \dyad_0\right] \ast\check\equi_n \|_\infty & = \| \left[ f(2^k \cdot) \tilde\dyad_0 g(2^k \cdot) \dyad_0 \right] \ast  \check\equi_n\|_\infty \\
& \lesssim \sum_{l,j}^* \| f(2^k \cdot) \tilde\dyad_0 \ast\check\equi_l \|_\infty \| g(2^k\cdot) \dyad_0 \ast \check\equi_j \|_\infty.
\end{align*}
Thus, calling $f_{k,l} = \langle l \rangle^\alpha  \|f(2^k \cdot) \tilde\dyad_0 \ast \equi_l\|_\infty$ and $g_{k,j} = \langle j \rangle^{\alpha} \|g(2^k \cdot)\dyad_0 \ast\check\equi_j\|_\infty,$ we have
\begin{align*}
\sum_{n \in \Z} \na^\alpha \sup_{k \in \Z} \| \left[ f(2^k \cdot) g(2^k\cdot) \dyad_0 \right] \ast\check\equi_n \|_\infty
& \lesssim \sum_{n \in \Z} \na^\alpha \sup_{k \in \Z} \sum_{l,j}^* \langle l \rangle^\alpha  \|f(2^k \cdot) \tilde\dyad_0 \ast \check\equi_l\|_\infty \\ 
& \langle j \rangle^{\alpha} \|g(2^k \cdot)\dyad_0 \ast\check\equi_j\|_\infty
\langle l \rangle^{-\alpha} \langle j \rangle^{-\alpha} \\
& \lesssim \sum_{n \in \Z} \langle n \rangle^\alpha \sum_{l,j}^* \langle l \rangle^{-\alpha} \langle j \rangle^{-\alpha} \sup_{k \in \Z} f_{k,l} g_{k,j} \\
& \lesssim \sum_{l \in \Z} \sup_k f_{k,l} \sum_{j \in \Z} \sup_{k} g_{k,j} \\
& \cong \|f\|_{\Eau} \|g\|_{\Eau},
\end{align*}
using the first part of the lemma in the end.
\end{proof}

We use the space $\Eau$ as a functional calculus space, as is also the case for $\Ma.$
We have the following embeddings between the two.

\begin{prop}
For any $\epsilon > 0,$ we have
$\M^{\alpha + 1 + \epsilon} \hookrightarrow \Eau \hookrightarrow \M^{\alpha - \epsilon}.$
\end{prop}

\begin{proof}
Start with the second embedding.
We have, using the compact support of $\dyad_0$ in the first line, and \cite[Proposition 3.5 (1)]{Kr1} in the second line,
\begin{align*}
\|f\|_{\M^{\alpha-\epsilon}} & 
\lesssim \sup_{k \in \Z} \| f(2^k \cdot) \dyad_0 \|_{\M^{\alpha-\epsilon}}
\lesssim \sup_{k \in \Z} \| f(2^k \cdot) \dyad_0 \|_{B^\alpha_{\infty,1}} \\
& \lesssim \sup_{k \in \Z} \| f(2^k \cdot) \dyad_0\|_{\Eai}
= \sup_{k \in \Z} \sum_{n \in \Z} \na^\alpha \| f(2^k \cdot) \dyad_0 \ast \check\equi_n \|_\infty  \\
& \leq \sum_{n \in \Z} \na^\alpha \sup_{k \in \Z} \|f(2^k \cdot) \dyad_0 \ast \check\equi_n\|_\infty = \|f\|_{\Eau}.
\end{align*}
For the first embedding, let for $n \in \N,$ $A_n = \{ k \in \N:\: 2^{n-1} \leq k \leq 2^n - 1\},\:A_{-n} = -A_n$ and $A_0 = \{ 0 \}.$
Thus the $A_n$ form a disjoint partition of $\Z$.
Let $(\Fdyad_n)_{n \in \Z}$ be a dyadic Fourier partition of unity.
Then 
\begin{align*}
\|f\|_{\Eau} & = \sum_{n \in \Z} \na^\alpha \sup_{k \in \Z} \| [f(2^k\cdot) \dyad_0 ]\ast \check\equi_n\|_\infty
= \sum_{n \in \Z} \sum_{l \in A_n} \langle l \rangle^\alpha \sup_{k \in \Z} \| \left[ f(2^k \cdot) \dyad_0 \right] \ast \check\equi_l \ast \tilde\Fdyad_n\check{\phantom{i}} \|_\infty \\
& \lesssim \sum_{n \in \Z} 2^{|n|\alpha} 2^{|n|} \sup_{k \in \Z}  \| [f(2^k \cdot)\dyad_0 \ast \tilde\Fdyad_n\check{\phantom{i}} \|_\infty \\
& = \sum_{n \in \Z} 2^{-|n| \epsilon} \sup_{k \in \Z} \left( 2^{|n| (\alpha + \epsilon + 1)} \| [f(2^k \cdot) \dyad_0] \ast \tilde\Fdyad_n\check{\phantom{i}} \|_\infty \right) \\
& \leq \sum_{n \in \Z} 2^{-|n| \epsilon} \sup_{k \in \Z} \sup_{m \in \Z} 2^{|m| (\alpha + 1 +\epsilon)} \| [f(2^k \cdot) \dyad_0 ] \ast \check\Fdyad_m\|_\infty \\
& \lesssim \sup_{k \in \Z} \|f(2^k\cdot)\dyad_0 \|_{B^{\alpha + 1 + \epsilon}_{\infty,\infty}} \\
& \lesssim \|f\|_{\M^{\alpha + 1 + \epsilon'}},
\end{align*}
using again the compact support of $\dyad_0$ in the last line.
\end{proof}

The following proposition of transference principle type is the main result of this section.
It can be compared to \cite[Theorem 4.9]{Kr1}.

\begin{prop}\label{Prop 1}
Let $A$ be a $0$-sectorial operator such that 
\[\|(1+A)^{-\beta_1} e^{itA} \| \leq C \ta^\alpha\]
and 
\[\{  ( 1 + A )^{-\beta_2} e^{itA} :\: t \in [0,1] \}\text{ is }\gamma\text{-bounded,}\]
for some constants $\beta_1,\,\beta_2 \geq \alpha > 0.$
Then $A$ has a smoothed $\Eai$ functional calculus in the sense that for $\beta = \beta_1 + 2 \beta_2,$
\[ \|(1 + A)^{-\beta} f(A) \| \leq C \| f \|_{\Eai} \quad (f \in \Eai,\:f\text{ has compact support in }(0,\infty)). \]
\end{prop}

\begin{proof}
Assume first that $f \in C^\infty_c(0,\infty).$
Then we have by a representation formula \cite[Lemma 4.77]{Kr}
\begin{align} (1 + A)^{-\beta} f(A) x & =  \frac{1}{2\pi} \int_{\R} \hat{f}(t) (1 + A)^{-\beta} e^{itA} x dt \nonumber \\ 
& = \frac{1}{2\pi} \int_\R \sum_{n \in \Z} \hat{f}(t) \equi_n(t) (1+A)^{-\beta} e^{itA} x dt.\label{Equ Proof Prop 1}
\end{align}
Write $ I : X \to \gamma(\R,X),\: x \mapsto 1_{[n-2,n+1]}(-t) (1+A)^{-\beta_1-\beta_2} e^{-itA}x$
and $P : \gamma(\R,X) \to X,\: g \mapsto \int 1_{[0,1]} (1 + A)^{-\beta_2} e^{itA}g(t) dt.$
Further, we let $M_{\hat{f} \equi_n} : \gamma(\R,X) \to \gamma(\R,X)$ be the convolution with $\hat{f} \equi_n.$
Recall that the Fourier transform is isometric on $L^2(\R),$ so by Lemma \ref{Lem gamma lifting} (1) also on $\gamma(\R,X).$
We thus have by \cite[Proof of Proposition 4.6 (2)]{Kr1} that $\|M_{\hat{f} \equi_n} \|_{\gamma(\R,X) \to \gamma(\R,X)} \cong \| f \ast \check\equi_n \|_\infty.$
One easily checks that
\[ \eqref{Equ Proof Prop 1} = \frac{1}{2\pi} \sum_{n \in \Z} P M_{\hat{f}\equi_n} I (x).\]
Note that 
\[ \| I : X \to \gamma(\R,X) \| \lesssim \gamma(\{ (1+A)^{-\beta_1 - \beta_2} e^{-itA} :\: t \in [n-2,n+1] \} ) \leq C \na^\alpha, \]
and by Lemma \ref{Lem gamma lifting} (2) also \[\|P\| \leq \gamma(\{ 1 + A)^{-\beta_2} e^{itA} :\: t \in [0,1] \} ) < \infty.\]
We conclude $\|(1 + A)^{-\beta} f(A) x \| \leq C \sum_{n \in \Z} \na^\alpha \| f \ast \check\equi_n\|_\infty \| x \| \cong \|f\|_{\Eai} \|x\|.$
The proposition follows since $C^\infty_c(0,\infty)$ is dense in $\{ f \in \Eai:\:f \text{ has compact support in }(0,\infty) \}.$
For example, the reader may check that $\| f \ast \rho_m - f \|_{\Eai} \to 0$ for any sequence $(\rho_m)_m \subset C^\infty_c(\R)$ with $\supp \rho_m \subset (-\frac1m,\frac1m),\:\int_\R \rho_m = 1,\:\rho_m \geq 0.$
\end{proof}

\begin{rem}
Note that the second hypothesis of Proposition \ref{Prop 1} is satisfied for any operator having a bounded Mihlin calculus \cite[Theorem 4.73]{Kr}.
Then the above proposition applies in two cases.
Firstly, if $A = (-\Delta)^{\frac12}$ on $X = L^p(\R^d)$ for some $1 < p < \infty,$
then the hypotheses are satisfied for any $\alpha > \frac{d-1}{2}$  \cite{Per}.
Secondly, if $A$ is the square root of a sublaplacian on the Heisenberg group, then the hypotheses are also satisfied for any $\alpha> \frac{d-1}{2}$ \cite[(3.1)]{Mu}.
Note that the critical order $\frac{d-1}{2}$ is by $\frac12$ smaller, so better, than
the critical order of $\frac{d}{2}$ in usual spectral multiplier theory.
\end{rem}

\section{$\Eau$ calculus}

Let $A$ be a $0$-sectorial operator.
Consider the conditions
\begin{equation}\label{Equ II.1}
 \gamma\left(\left\{ ( 1 + 2^k A)^{-\beta} e^{it2^k A} : \: k \in \Z \right\} \right) \leq C \ta^\alpha
\end{equation}
and
\begin{equation}\label{Equ II.2}
 \gamma\left(\left\{ ( 1 + 2^k A)^{-\gamma} e^{it2^k A} :\: k \in \Z,\:t \in [0,1] \right\} \right) < \infty.
\end{equation}

\begin{lem}\label{Lem 2}
Let $X$ have property $(\alpha)$ and $A$ be a $0$-sectorial operator satisfying \eqref{Equ II.1} and \eqref{Equ II.2}.
Let $G \subset \Eai$ such that any $f \in G$ has compact support in $(0,\infty).$
Then \[\{ (1+2^k A)^{-(\beta + 2 \gamma)} f(2^k A) :\: k \in \Z ,\: f \in G \}\text{ is }\gamma\text{-bounded} \]
provided $\sum_{n \in \Z} \na^\alpha \sup_{f \in G} \|f \ast \check\equi_n \|_\infty < \infty.$
\end{lem}

\begin{proof}
Let $\tilde{A} = \sum_{k \in \Z} 2^k P_k \otimes A$ be the operator defined on $\Gauss(X)$ where $P_k(\sum_{j \in \Z} \gamma_j \otimes x_j) = \gamma_k \otimes x_k,$
so that $\tilde{A} (\sum_{j \in \Z} \gamma_j \otimes x_j ) = \sum_{k \in \Z} \gamma_k \otimes 2^k A x_k.$
Put $\tilde{S}_\beta(t) = \sum_{k \in \Z} P_k \otimes ( 1+ 2^k A)^{-\beta}e^{it2^k A} = (1 + \tilde{A})^{-\beta} e^{it\tilde{A}}.$
Then \eqref{Equ II.1} $\Longleftrightarrow \|\tilde{S}_\beta(t) \| \lesssim \ta^\alpha$ and \eqref{Equ II.2} $\Longleftrightarrow \{ \tilde{S}_{\gamma}(t) : \: t \in [0,1] \}$ is $\gamma$-bounded in
$B(\Gauss(X)).$
Indeed, let $y_n \in \Gauss(X),\,t_n \in [0,1]$ and write $y_n = \sum_k \gamma_k \otimes x_{nk}.$
Then using property $(\alpha),$ and writing $S_\gamma^k(t) =(1 + 2^k A)^{-\gamma} e^{it2^k A},$
we have 
\begin{align*}
\| \sum_{n \in \Z} \gamma_n \otimes \tilde{S}_\gamma(t_n)y_n \|_{\Gauss(\Gauss(X))} & \cong \| \sum_{n,k} \gamma_{nk} \otimes S_\gamma^k(t_n)x_{nk}\|_{\Gauss(X)} \\
& \leq C \| \sum_{n,k} \gamma_{nk} \otimes x_{nk} \| \cong \| \sum_{n} \gamma_n \otimes y_n \|.
\end{align*}
Therefore, Proposition \ref{Prop 1} can be applied to the operator $\tilde{A}$ in place of $A$ and one obtains 
\[ \|( 1 + \tilde{A} )^{-(\beta + 2 \gamma)} f(\tilde{A}) \| \leq C \| f \|_{\Eai}. \]
Moreover, let $G$ satisfy the assumption of the lemma and $f_1,\ldots,f_N \in G.$
Put $f(t) = \sum_{k=1}^N \gamma_k \otimes f_k(t) \Id_X,$ so that $f : \R \to B(\Gauss(X)).$
The image of $f$ commutes with $\tilde{S}_\beta(t)$ for any $t \in \R.$
As in \cite[Proof of Proposition 5.5]{Kr1} it follows now from Proposition \ref{Prop 1}
that
\[\| ( 1 + \tilde{A} )^{-(\beta + 2 \gamma)} f(\tilde{A}) \| \lesssim \sum_{n \in \Z} \na^\alpha \gamma\left(\left\{f \ast \check\equi_n(t) :\: t \in \R \right\}\right) 
\lesssim \sum_{n \in \Z} \na^\alpha \sup_{f \in G} \|f \ast \check\equi_n \|_\infty,\]
where we used Kahane's contraction principle in the last step.
But $\| ( 1 + \tilde{A} )^{-(\beta + 2 \gamma)} f(\tilde{A}) \| = \gamma(\{ (1 + 2^k A)^{-(\beta + 2 \gamma)}f_l(2^kA) :\: k \in \Z,\,l= 1,\ldots,N \}),$
so the lemma follows by taking the supremum over all $f_1,\ldots,f_N \in G.$
\end{proof}


\begin{lem}\label{Lem 1}
Let $A$ be a $0$-sectorial operator. Let the conclusion of Lemma \ref{Lem 2} hold, i.e.
$\{ (1+2^k A)^{-(\beta + 2 \gamma)} f(2^k A) :\: k \in \Z ,\: f \in G \}$ is $\gamma$-bounded if $\sum_{n \in \Z} \na^\alpha \sup_{f \in G} \|f \ast \check\equi_n \|_\infty < \infty.$
Suppose that $A$ admits a Paley-Littlewood spectral decomposition.
That is, for a dyadic partition of unity $(\dyad_k)_{k\in \Z},$ we have
$\|x\| \cong \| \sum_{ k \in \Z} \gamma_k \otimes \dyad_k(A) x \|_{\Gauss(X)}.$
\begin{enumerate}
\item If $f \in \Eau,$ then $f(A) \in B(X).$
\item If $X$ has property $(\alpha)$ and $G\subset \Eai$ satisfies $\sum_{n \in \Z} \na^\alpha \sup_{f \in G} \sup_{k \in \Z} \|f(2^k \cdot)\dyad_0 \ast \check\equi_n\|_\infty < \infty,$
then $\{f(A) : \: f \in G \}$ is $\gamma$-bounded.
\end{enumerate}
\end{lem}

\begin{proof}
The first part of the lemma follows from the proof of the second part by considering $G = \{ f \}$ a singleton.
So let $G$ satisfy the hypotheses in (2) of the lemma and $f_1,\ldots,f_N \in G.$
Then by the Paley-Littlewood spectral decomposition and property $(\alpha),$
\[ \| \sum_{n = 1}^N \gamma_n \otimes f_n(A) x \| \cong \| \sum_{n,k} \gamma_{nk} \otimes (f_n \dyad_k)(A) \widetilde\dyad_k(A) x \| .\]
It thus remains to check that $\{ (f_n \dyad_k)(A) :\: n = 1 ,\ldots,N,\: k \in \Z \}$ is $\gamma$-bounded.
We have $(f_n \dyad_k)(A) = (f_n \dyad_0(2^{-k}\cdot))(A) = (f_n(2^k \cdot) \dyad_0)(2^{-k}A).$
Let $\tilde{G} = \{ f_n(2^k \cdot) \dyad_0 (1 + (\cdot))^{\beta + 2 \gamma} :\: n,k \}.$
Note that functions in $\tilde{G}$ have compact support in $(0,\infty)$.
If 
\begin{equation}\label{Equ Proof Lem 1}
\sum_n \na^\alpha \sup_{g \in \tilde{G}} \| g \ast \check\equi_n \|_\infty < \infty, 
\end{equation}
then
\[ \left\{ (1 + 2^l A)^{-(\beta + 2 \gamma)} g(2^l A) :\: l \in \Z,\:g \in \tilde{G} \right\} \supset \left\{(1+2^{-k}A)^{-(\beta+2 \gamma)} f_n(A) \dyad_k(A) ( 1 + 2^{-k} A)^{\beta + 2 \gamma} :\: n,k \right\} \]
would be $\gamma$-bounded and the lemma would follow.
It remains to show \eqref{Equ Proof Lem 1}.
Denoting $\sum^*_{l,j} = \sum_{l,j:\:|n-l-j| \leq 3},$ we have
\begin{align*}
& \sum_{n \in \Z} \na^\alpha \sup_{m \leq N,\: k \in \Z} \| f_m(2^k \cdot)(1 + \cdot)^{\beta + 2 \gamma} \dyad_0 \ast \check\equi_n \|_\infty \\
& \leq \sum_{n \in \Z} \na^\alpha \sum_{l,j}^* \| \left[ \left( f_m(2^k \cdot)\dyad_0 \ast \check\equi_l \right) \left( ( 1 + \cdot)^{\beta + 2 \gamma}\tilde\dyad_0 \ast \check\equi_j \right) \right] \ast \check\equi_n \|_\infty. \\
& \leq \sum_{n \in \Z} \na^\alpha \sum_{l,j}^* \| f_m(2^k \cdot)\dyad_0 \ast \check\equi_l \|_\infty \| (1 + \cdot)^{\beta + 2 \gamma} \tilde\dyad_0 \ast \check\equi_j \|_\infty \| \check\equi_n\|_1 \\
& \leq \sum_{n \in \Z} \na^\alpha \sum_{l,j}^* \langle j \rangle^{-\beta'} \| (1 + \cdot)^{\beta + 2 \gamma} \tilde\dyad_0 \|_{E^{\beta'}_\infty} \sup_{k,m} \| f_m(2^k \cdot)\dyad_0 \ast \check\equi_l \|_\infty
\end{align*}
where we choose $\beta'>\alpha + 1.$
Then the above inequalities continue
\begin{align*}
& = \sum_{l \in \Z} \langle l \rangle^\alpha \sum_{n,j}^* \langle l \rangle^{-\alpha} \langle n \rangle^\alpha \langle j \rangle^{-\beta'} \sup_{k,m} \| f_m(2^k \cdot)\dyad_0 \ast \check\equi_l \| \\
& \lesssim \sum_{l \in \Z} \langle l \rangle^\alpha \sup_{f \in G} \sup_{k \in \Z} \| \left( f(2^k \cdot) \dyad_0 \right) \ast \check\equi_l \|_\infty,
\end{align*}
which is finite according to the hypothesis.
\end{proof}

We are now able to prove the main result of this section which is the following theorem.

\begin{thm}\label{Thm 4}
Let $X$ have property $(\alpha).$
Assume that $A$ has a bounded $\M^\beta$ calculus for some $\beta$ and let $\alpha > 0$ be a parameter.
Then $(B) \Longrightarrow (A) \Longrightarrow (B'),$ where
\begin{enumerate}
\item[(A)] $\gamma\left(\{(1 + 2^k A)^{-\beta_1} e^{i2^k t A} :\: k \in \Z \} \right) \leq C \ta^\alpha $ for some $\beta_1 \geq \alpha.$
\item[(B)] $\gamma \left(\{(1 + 2^k A)^{-\beta_2} f(2^k A) :\: k \in \Z \} \right) \leq C \|f\|_{\Eai}$ for some $\beta_2 \geq \alpha$ and any $f \in \Eai.$
\item[(B')] $\gamma \left(\{(1 + 2^k A)^{-\beta_2} f(2^k A) :\: k \in \Z \} \right) \leq C \|f\|_{\Eai}$ for some $\beta_2 \geq \alpha$ and any $f \in \Eai$ with compact support in $(0,\infty).$
\end{enumerate}
Conditions $(A)$ and $(B)$ imply moreover that
\begin{enumerate}
\item[(C)] $\|f(A)\| \leq C \|f\|_{\Eau} \quad (f \in \Eau).$
\item[(D)] If $G \subset \Eau$ such that $\sum_{n \in \Z} \na^\alpha \sup_{f \in G} \sup_{k \in \Z} \| \left( f(2^k \cdot) \dyad_0 \right) \ast \check\equi_n \|_\infty < \infty,$
then $ \{ f(A) :\: f \in G \}$ is $\gamma$-bounded.
\end{enumerate}
\end{thm}

\begin{proof}
Since $A$ has a bounded $\M^\beta$ calculus, we have 
\begin{equation}\label{Equ Ap}\tag{A'}\{ ( 1 + 2^k A)^{-\gamma} e^{it 2^k A} :\: k \in \Z,\: t \in [0,1] \}\text{ is }\gamma\text{-bounded for }\gamma\text{ sufficiently large}
\end{equation}
\cite[Theorem 4.73]{Kr}.
Then Lemma \ref{Lem 2} shows that \eqref{Equ Ap} and (A) imply (B') with $\beta_2 = \beta_1 + 2 \gamma.$
On the other hand, (B) implies (A) with $\beta_1 = \beta_2$ because of $\|e^{it(\cdot)}\|_{\Eai} \lesssim \ta^\alpha$ \cite[Proof of Theorem 4.9]{Kr1}.
The bounded $\M^\beta$ calculus also implies that the Paley-Littlewood spectral decomposition $\|x\| \cong \| \sum_{ k \in \Z} \gamma_k \otimes \dyad_k(A) x \|$ holds
\cite{KrW}.
Then \eqref{Equ Ap}, (A) (resp. (B)) and the Paley-Littlewood decomposition show with Lemma \ref{Lem 1} that (C) and (D) hold.
\end{proof}

\section{Application: Poisson semigroup}\label{Sec 5 Poisson Semigroup}

We now apply Theorem \ref{Thm 4} to the square root of the standard Laplacian on $L^p(\R^d).$
That is, we check condition (A).

\begin{thm}\label{Thm Poisson}
Let $A = (- \Delta)^{\frac12}$ on $X = L^p(\R^d)$ for some $d \in \N$ and $1 < p < \infty,$
i.e. the generated semigroup $\exp(-e^{i\theta}tA)$ is analytic on the right halfplane and has as integral kernel the Poisson kernel
\[ p_{t,\theta}(x) = \frac{e^{i\theta}t}{((e^{i\theta}t)^2 + |x|^2)^{\frac{d+1}{2}}}. \]
Then for any $\alpha > \frac{d-1}{2},$ $\{ \exp(-e^{i\theta} t 2^k A) :\: k  \in \Z \}$ is $\gamma$-bounded with bound
$\lesssim (\frac{\pi}{2} - |\theta|)^{-\alpha}$ for any  $|\theta| < \frac{\pi}{2}.$
Consequently, by \eqref{Equ Wave Bound},
\begin{equation}\label{Equ Poisson Thm}
\gamma \left( \{ (1 + 2^k A)^{-\alpha} e^{it 2^k A} : \: k \in  \Z \} \right) \lesssim \ta^\alpha,
\end{equation}
so condition \eqref{Equ II.1} is satisfied for any $\alpha > \frac{d-1}{2}.$
\end{thm}

\begin{proof}
Our proof follows closely the chapter on maximal functions in \cite{St}.
Note that on $L^p$ spaces for $p < \infty,$ one has $\| \sum_k \gamma_k \otimes x_k \|_p \cong \| \left( \sum_k |x_k|^2 \right)^{\frac12} \|_p.$
Thus according to \cite[p.~76, 5.4]{St} it suffices to show that for any $\alpha > \frac{d-1}{2}$
\begin{equation}\label{Equ 1 Proof Theorem}
\int_{|x| \geq 2 |y|} |p_{2^kt,\theta}(x-y)-p_{2^k t,\theta}(x)| dx \leq C (\frac{\pi}{2} - |\theta|)^{-\alpha} \quad (k \in \Z).
\end{equation}
According to the proof in \cite[p.~74]{St}, \eqref{Equ 1 Proof Theorem} follows from the hypotheses of \cite[4.2.1 Corollary]{St}.
This means that it remains to show
\begin{align}
\int | \Phi(x-y) - \Phi(x) | dx & \leq \eta(|y|)\label{Equ 2 Proof Theorem}
\intertext{and}
\int_{|x| \geq R} |\Phi(x)| dx & \leq \eta(R^{-1}) ,\: R \geq 1 \label{Equ 3 Proof Theorem}
\end{align}
for some Dini modulus $\eta,$ i.e. $\int_0^1 \eta(y) \frac{dy}{y} < \infty,$ and for $\Phi(x) = |p_{t,\theta}(x)|.$
According to \cite[p.~74]{St}, $\int_0^1 \eta(y) \frac{dy}{y}$ is then an upper bound for the $\gamma$ bound in the claim \eqref{Equ Poisson Thm}.
For simplicity suppose first $t = 1.$
We write 
\[ C_2(s) = \int_{\R^d} | \nabla \phi^{(s)}(x) | dx ,\: C_3(s) = \int_{\R^d} | \Phi^{(s)}(x) | dx,\: C_4 = \int_{\R^d} |\Phi(x)| (1 + |x|)^\delta dx,\]
where $\Phi^{(s)}(x) = |e^{2i\theta} + x^2|^{-\frac{d+1}{2}(1-s)} ( 1 + |x| )^{-cs},$ and where $c,\delta$ are positive constants.
The $\Phi^{(s)}$ form a family analytic in $s$ with $\Phi^{(0)} = \Phi.$
Thus by the three lines lemma
\[ \int | \Phi^{(0)}(x-y) - \Phi^{(0)}(x) | dx \lesssim C_3(-\epsilon)^{1-\vartheta} C_2(1)^\vartheta |y|^\vartheta \]
for the parameter $\vartheta$ given by $0 = - \epsilon (1 - \vartheta) + 1 \cdot \vartheta,$ so $\vartheta = \frac{\epsilon}{1 + \epsilon} \in (0,1).$
Concerning \eqref{Equ 3 Proof Theorem}, if $C_4 = \int |\Phi(x)| (1 + |x|)^\delta dx < \infty,$ for some $\delta > 0$ then
$\int_{|x| \geq R} |\Phi(x)| dx = \int_{|x| \geq R} |\Phi(x)| (1 + |x|)^\delta (1 +|x|)^{-\delta} dx \leq (1+R)^{-\delta} C_4.$
So choosing $\eta(u) = cu^\beta (C_4 + C_3(-\epsilon)^{1-\vartheta} C_2(1)^\vartheta)$ with $\beta = \min(\frac{\epsilon}{1 +\epsilon},\delta)$,
we have the estimate
\[ \int_0^1 \eta(u) \frac{du}{u} \lesssim \frac{1}{\beta} (C_4 + C_3(-\epsilon)^{1-\vartheta} C_2(1)^\vartheta) . \]
Let us now estimate the expressions $C_2,C_3,C_4.$
We have $C_4 = \int |e^{2 i \theta} + x^2 |^{-\frac{d+1}{2}} (1 + |x|)^\delta dx = \int | \cos(2 \theta) + x^2 + i \sin(2\theta)|^{-\frac{d+1}{2}} (1 +|x|)^\delta dx.$
The integrand is radial, and depending on the radius, the real or the imaginary part dominates.
If $| |x|^2 - 1 | \geq \frac{\pi}{2} - |\theta|,$ then the real part dominates, otherwise the imaginary part dominates.
Thus we naturally divide the integral $C_4$ into the three regions
$0 \leq x^2 \leq 1 - (\frac{\pi}{2} - |\theta|),\: 1-  (\frac{\pi}{2} - |\theta|) \leq x^2 \leq 1 +  (\frac{\pi}{2} - |\theta|),$ and $1+ (\frac{\pi}{2} - |\theta|) \leq x^2.$
Then a simple calculation shows
\begin{align*} C_4 & \cong \int_0^\infty |\cos(2 \theta) + s + i \sin(2 \theta)|^{-\frac{d+1}{2}} ( 1 + s )^{\frac{\delta}{2}} s^{\frac{d}{2}} \frac{ds}{s} \\
& \cong \int_0^{1 - (\frac{\pi}{2} - |\theta|)} \ldots + \int_{1 - (\frac{\pi}{2} - |\theta|)}^{1 + (\frac{\pi}{2} - |\theta|)} \ldots + \int_{1 + \frac{\pi}{2} - |\theta|}^\infty \ldots \\
& \cong 1 + (\frac{\pi}{2} - |\theta|)^{-\frac{d-1}{2}} + (\frac{\pi}{2} - | \theta |)^{-{\frac{d-1}{2}}} + (\frac{\pi}{2} - | \theta |)^{-{\frac{d-1}{2}}}
\end{align*}
as soon as the parameter $\delta < 1.$
Let us turn to $C_2.$
We have with $P(x) = e^{2 i \theta} + x^2$ and $a = \frac{d+1}{2}$
\begin{align*}
\nabla \Phi^{(s)} (x) & = - a (1-s) |P(x)|^{-a(1-s)-1} \frac{\Re P(x)}{|P(x)|} P'(x) ( 1 + |x| )^{-cs}
+ (-cs) \frac{x}{|x|} |P(x)|^{-a(1-s)} ( 1 + |x| )^{-cs - 1} \\
& = |P(x)|^{-a(1-s)-1} (1 + |x|)^{-cs-1} \cdot \left( - a (1-s) \frac{\Re P(x)}{|P(x)|} P'(x)(1 + |x|) - cs \frac{x}{|x|} |P(x)| \right).
\end{align*}
If $\Re s = 1,$ then $P'(x) = 2x,$ and the first term in the above brackets is dominated by $\lesssim |\Im s |\cdot|x|\cdot(1+|x|).$
In all we get for $\Re s =  1$
\[ C_2(s) \lesssim \int_{\R^d} | \Im s | \cdot |x| \cdot (1 + |x|)^{-c} |P(x)|^{-1} dx + \int_{\R^d} |s| \cdot (1 + |x|)^{-c-1} dx = : C_2^{(1)} + C_2^{(2)}.\]
We have $C_2^{(1)} \cong \int_0^\infty | \Im(s) | r ( 1 + r )^{-c} |P(r)|^{-1} r^{d-1} dr \cong \int_0^2 |\Im(s)| r |P(r)|^{-1} r^{d-1} dr + \int_2^\infty | \Im(s) | r (1+r)^{-c} |P(r)|^{-1} r^{d-1} dr
\lesssim |\Im(s)| (\frac{\pi}{2} - |\theta|)^{-1}$ for $c > d-1.$
On the other hand, $C_2^{(2)} < \infty$ as soon as $c$ is large enough ($c > d-1$).
In all, $C_2(s) \lesssim |\Im(s)| (\frac{\pi}{2} - |\theta|)^{-1}.$

Let us finally turn to $C_3(s).$
We consider $\Re s = - \epsilon < 0.$
Then 
\begin{align*}
C_3(s) & = \int |P(x)|^{-a(1+\epsilon)} (1 + |x|)^{c \epsilon} dx \cong \int_0^\infty |P(r)|^{-a(1+\epsilon)} (1 + r)^{c \epsilon} r^{d-1} dr
& \cong \int_0^2 \ldots + \int_2^\infty \ldots.
\end{align*}
The first integral can be estimated against $\lesssim (\frac{\pi}{2} - |\theta|)^{-a(1+\epsilon) + 1},$ and the second integral is finite as soon as $\epsilon(c-d-1)< 1.$
In all, we get $\int_0^1 \eta(y) \frac{dy}{y} \lesssim (C_4 + C_3(-\epsilon)^{1-\vartheta} C_2(1)^\vartheta) \lesssim (\frac{\pi}{2} - |\theta|)^{-\frac{\epsilon}{1+\epsilon} - a + \frac{1}{1+\epsilon}} 
\cong (\frac{\pi}{2} - |\theta|)^{-\frac{d-1}{2} + \tilde{\epsilon}},$ with $\frac{1-\epsilon}{1 + \epsilon} = 1 - \tilde{\epsilon}.$
Now it is easy to repeat the argument for $p_{t,\theta}$ in place of $p_{1,\theta}.$
This finishes the proof.
\end{proof}

Theorem \ref{Thm Poisson} can be used in combination with Theorem \ref{Thm 4}, but moreover it has also a consequence for the Mihlin functional calculus of $-\Delta.$
Note that the classical theorem of Mihlin gives mere boundedness of the set in \eqref{Equ Cor 1} below.

\begin{cor}
The operator $A = - \Delta$ on $L^p(\R^d)$ for $1 < p < \infty$ has a Mihlin calculus satisfying
\begin{equation}\label{Equ Cor 1}
\{ f(A) :\: \|f\|_{\Ma} \leq 1  \} \text{ is }\gamma\text{-bounded}
\end{equation}

for $\alpha > \frac{d}{2}.$
\end{cor}

\begin{proof}
This follows from \cite[Proposition 4.79]{Kr} applied to the estimate
\[\gamma \left(\{ (1 + 2^k A)^{-\beta} \exp(i2^ktA) :\: k \in \Z \} \right) \lesssim \ta^{\beta} \]
with $\beta > \frac{d-1}{2}$ and $A = (-\Delta)^{\frac12}.$
Note that for the underlying $L^p$ space, one always has $\frac{1}{\text{type }X} - \frac{1}{\text{cotype }X} < \frac12,$ and 
$(-\Delta)^{\frac12}$ has a $\Ma$ calculus because $-\Delta$ has a $\Ma$ calculus, for, say, $\alpha > \frac{d}{2},$ so admits the Paley-Littlewood decomposition \eqref{Equ Paley-Littlewood}.
\end{proof}


\end{document}